\documentclass[reqno]{amsart}%
\usepackage[utf8]{inputenc}%
\usepackage[english]{babel}%
\usepackage{amsmath,amssymb,amsthm,amsfonts}%
\usepackage[foot]{amsaddr}%
\usepackage{hyperref}%
\usepackage{enumerate}%
\usepackage{MnSymbol}
\usepackage{bbold}
\allowdisplaybreaks%
\newcommand{\la}{\lambda}

\newcommand{\al}{\alpha}
\newcommand{\si}{\sigma}

\newcommand{\Z}{\mathbb{Z}}
\newcommand{\C}{\mathbb{C}}
\newcommand{\Sb}{\mathbb{S}}
\newcommand{\Yb}{\mathbb{Y}}
\newcommand{\R}{\mathbb{R}}
\newcommand{\Ply}{\mathrm{Pl}}
\newcommand{\Pls}{\mathbb{Pl}}
\newcommand{\xy}{\mathrm{x}}
\newcommand{\yy}{\mathrm{y}}
\newcommand{\xs}{\mathbb{x}}
\newcommand{\ys}{\mathbb{y}}
\newcommand{\My}{\mathrm{M}}
\newcommand{\Ms}{\mathbb{M}}
\newcommand{\Dc}{\mathcal{D}}
\newcommand{\Rc}{\mathcal{R}}
\newcommand{\prob}{\mathop{\mathsf{Prob}}}
\renewcommand{\i}{\mathsf{i}}
\renewcommand{\j}{\mathsf{j}}
\renewcommand{\c}{\mathsf{c}}
\newcommand{\da}{\downarrow}
\newcommand{\Da}{\Downarrow}
\newcommand{\ua}{\uparrow}
\newcommand{\Ua}{\Uparrow}
\newtheorem{prop}{Proposition}
\newtheorem{lemma}[prop]{Lemma}

\theoremstyle{definition}

\begin{document}
\title[]{On Measures on Partitions Arising in\\ Harmonic Analysis for Linear and Projective Characters of the Infinite Symmetric Group}
\author{Leonid Petrov}
\thanks{The author was partially supported by the RFBR-CNRS grant  10-01-93114 and by the Dynasty foundation fellowship for young scientists}
\address{Dobrushin Mathematics Laboratory,
Kharkevich Institute for Information Transmission Problems,
Bolshoy Karetny per.~19, Moscow, 127994, Russia}
\email{lenia.petrov@gmail.com}

\begin{abstract}  
  The $z$-measures on partitions originated from the problem of harmonic analysis of linear representations of the infinite symmetric group in the works of Kerov, Olshanski and Vershik \cite{Kerov1993}, \cite{Kerov2004}. A similar family corresponding to projective representations was introduced by Borodin \cite{Borodin1997}. The latter measures live on strict partitions (i.e., partitions with distinct parts), and the $z$-measures are supported by all partitions. In this note we describe some combinatorial relations between these two families of measures using the well-known doubling of shifted Young diagrams.
\end{abstract}

\maketitle

\section{Ordinary and strict partitions}

A \emph{partition} is an integer sequence of the form $\rho=(\rho_1\ge\ldots\ge\rho_{\ell(\rho)},0,0,\dots)$, where each $\rho_i>0$ and only finitely many of them are nonzero. A partition is called \emph{strict} if all its nonzero parts are distinct. Strict partitions are denoted by $\la,\mu,\dots$. Partitions which are not necessary strict will be called \emph{ordinary} and denoted by $\rho,\si,\dots$. We denote $|\rho|:=\rho_1+\dots+\rho_{\ell(\rho)}$, this is the \emph{weight} of a partition. Set $\Yb_n:=\{\rho\colon|\rho|=n\}$, $\Sb_n:=\{\la\colon\mbox{$\la$ strict and $|\la|=n$}\}$, $n=0,1,2,\dots$ (by agreement, $\Yb_0=\Sb_0=\{\varnothing\}$).

We identify ordinary and strict partitions with corresponding \emph{ordinary} and \emph{shifted Young diagrams}, respectively \cite[I.1]{Macdonald1995}. For example: 
\begin{equation}
  \rho=(4,4,1)\longleftrightarrow
  \begin{array}{|c|c|c|c|}
    \hline
    \ &\ &\ &\ \\
    \hline 
    \ &\ &\ &\ \\
    \cline{1-4}
    \ \\
    \cline{1-1}
  \end{array}
  \qquad\qquad
  \la=(5,3,2)\longleftrightarrow
  \begin{array}{|c|c|c|c|c|}
    \hline
    \ &\ &\ &\ &\ \\
    \hline 
    \multicolumn{1}{c|}\ &\ &\ &\ \\
    \cline{2-4}
    \multicolumn{2}{c|}\ &\ &\ \\
    \cline{3-4}
  \end{array}
  \label{YS_diagrams}
\end{equation}
For any box $\square$ in an ordinary of shifted Young diagram, by $\i(\square)$ and $\j(\square)$ we denote its row and column numbers, respectively. Also $\c(\square):=\j(\square)-\i(\square)$ is the \emph{content} of the box. Clearly, the content of any box in a shifted Young diagram is nonnegative.

If we have $|\rho|=|\si|+1$ and $\si\subset\rho$ for ordinary Young diagrams $\si$, $\rho$ (i.e., $\rho$ is obtained from $\si$ by adding a box), then we write $\si\nearrow\rho$, or, equivalently, $\rho\searrow\si$. In a similar situation for shifted diagrams $\mu$, $\la$ we write $\mu\Nearrow\la$ or $\la\Searrow\mu$. 

The \emph{Young graph} $\Yb=\bigsqcup_{n=0}^\infty\Yb_n$ consists of all ordinary Young diagrams, and we connect $\si\in\Yb_{n-1}$ and $\rho\in\Yb_n$ by an edge iff $\si\nearrow\rho$. This is a graded graph which describes the branching of irreducible representations of the symmetric groups $\mathfrak{S}(n)$, see \cite[I.7]{Macdonald1995} or \cite{Vershik_Okounkov1996}. The \emph{Schur graph} $\Sb=\bigsqcup_{n=0}^\infty\Sb_n$ is defined in the same manner for shifted Young diagrams. This graded graph describes the branching of (suitably normalized) irreducible truly projective characters of the symmetric groups $\mathfrak{S}(n)$ \cite{Hoffman1992}, \cite{IvanovNewYork3517-3530}.

For $\rho\in\Yb$, by $f_\rho$ denote the number of paths in $\Yb$ from the initial vertex $\varnothing$ to the diagram $\rho$. The number of paths in the Schur graph from $\varnothing$ to $\la\in\Sb$ is denoted by $g_\la$. There are explicit formulas for $f_\rho$ and $g_\la$ \cite[I.5, III.8]{Macdonald1995}.

\section{Coherent systems of measures}

\subsection{Young graph}
\emph{Down transition probabilities} on the Young graph are
\begin{equation}\label{down_Y}
  p^\da(\rho,\si):=\begin{cases}
    f_\si/f_\rho,&\mbox{if $\si\nearrow\rho$},\\
    0,&\mbox{otherwise}.
  \end{cases}
\end{equation}
One sees that $p^\da$ give rise to a Markov transition kernel from $\Yb_n$ to $\Yb_{n-1}$ (for any $n=1,2,\dots$), i.e., to a random procedure of deleting a box from an ordinary Young diagram. 

A sequence of probability measures $\{\My_n\}$ on $\Yb_n$ is called \emph{coherent} iff $\My_n$ is compatible with the down transition kernel $p^\da$: $\My_n\circ p^\da=\My_{n-1}$, or, in more detail,
\begin{equation*}
  \sum_{\rho\colon\rho\searrow\si} 
  \My_n(\rho)p^\da(\rho,\si)=\My_{n-1}(\si)
  \quad\mbox{for all $n=1,2,\dots$, and $\si\in\Yb_{n-1}$}.
\end{equation*}
We assume that our coherent systems are \emph{nondegenerate}, i.e., each $\My_n$ is supported by the whole $\Yb_n$. Having a nondegenerate coherent system $\{\My_n\}$, one can define the corresponding up transition kernel from $\Yb_n$ to $\Yb_{n+1}$ (for any $n=0,1,\dots$):
\begin{equation*}
  p^\ua(\si,\rho):=\begin{cases}
    p^\da(\rho,\si)\My_{n+1}(\rho)/\My_n(\si),&
    \mbox{if $\si\nearrow\rho$ and $|\si|=n$},\\
    0,&\mbox{otherwise}.
  \end{cases}
\end{equation*}
The up transition probabilities depend on the choice of a coherent system $\{\My_n\}$, and they define it uniquely. Moreover, $\My_n\circ p^\ua=\My_{n+1}$ for all $n$. In this way, $p^\ua$ define a random procedure of adding a box to a Young diagram. Iterating this procedure, one can think of a process of random growth of a diagram (by adding one box at a time) which starts from $\varnothing$. Then $\My_n$ is the distribution of a Young diagram after adding $n$ boxes. It is known (e.g., see \cite{vershik1987locally}) that linear characters of the infinite symmetric group $\mathfrak{S}(\infty)$ are in one-to-one correspondence with coherent systems on the Young graph.

The well-known Plancherel measures on ordinary partitions $\Ply_n(\rho)={f_\rho^2}/{n!}$ ($\rho\in\Yb_n$) form a distinguished coherent system $\{\Ply_n\}$ on the Young graph. It has the up transition probabilities $p^\ua_{\Ply}(\si,\rho)=\frac{f_\rho}{(|\si|+1)f_\si}$ ($\si\nearrow\rho$).

The problem of harmonic analysis on the infinite symmetric group \cite{Kerov1993}, \cite{Kerov2004} leads to a deformation $\{\My_n^{z,z'}\}$ of the Plancherel measures $\Ply_n$ depending on two complex parameters $z$ and $z'$ subject to the following constraints: 

$\bullet$ either $z'=\bar z$ and $z\in\C\setminus\Z$,

$\bullet$ or $z,z'\in\R$ and $m<z,z'<m+1$ for some $m\in\Z$.

\par\noindent
The system of deformed measures $\{\My_n^{z,z'}\}$ (they are called the \emph{$z$-measures}) is also coherent, and its up transition probabilities have the form (e.g., see \cite{Kerov2000}):
\begin{equation}\label{up_Z}
  p^\ua_{z,z'}(\si,\rho)=
  \frac{(z+\c(\square))(z'+\c(\square))}{zz'+|\si|}
  p^\ua_{\Ply}(\si,\rho),\qquad\si\nearrow\rho,\quad   
  \square=\rho\setminus\si.
\end{equation}
The $z$-measures $\{\My_n^{z,z'}\}$ is a remarkable object, they were studied in great detail by Borodin, Olshanski, Okounkov, and other authors.

\subsection{Schur graph}
General concepts explained above in the case of the Young graph work in the same way for the Schur graph. The down transition probabilities here are denoted by $p^\Da$, they are defined as in (\ref{down_Y}) using the quantities $g_\la$. Coherent systems of measures on the Schur graph correspond to (truly) projective characters of $\mathfrak{S}(\infty)$ (e.g., see \cite{nazarov1992factor}, \cite{IvanovNewYork3517-3530} and a general formalism of \cite{vershik1987locally}).

There are also Plancherel measures on strict partitions $\Pls_n(\la)=2^{n-\ell(\la)}g_\la^2/n!$ (here $\la\in\Sb_n$ and $\ell(\la)$ is the number of rows in the shifted diagram $\la$), they form a distinguished coherent system on $\Sb$. The corresponding up transition probabilities are
$p^\Ua_{\Pls}(\mu,\la)=\frac{g_\la}{(|\mu|+1)g_\mu} 2^{\ell(\mu)-\ell(\la)+1}$ ($\mu\Nearrow\la$). A deformation $\{\Ms_n^\al\}$ of the Plancherel measures $\Pls_n$  depending on one parameter $\al>0$ was introduced in \cite{Borodin1997}. The measures $\Ms_n^\al$ form a coherent system which can be described in terms of its up transition probabilities:
\begin{equation}\label{up_mult}
  p^\Ua_\al(\mu,\la)=\frac{\c(\square)
  \cdot(\c(\square)+1)+\al}{2|\mu|+\al}
  p^\Ua_{\Pls}(\mu,\la),
  \qquad\mu\Nearrow\la,\quad
  \square=\la\setminus\mu.
\end{equation}

The Plancherel measures on Young and Schur graphs admit a unified combinatorial description which can be read from, e.g., \cite{fomin1994duality}. The measures $\{\Ms_n^\al\}$ on the Schur graph do not have a representation-theoretic interpretation in the spirit of  \cite{Kerov1993}, \cite{Kerov2004} yet. However, combinatorially they look very similar to the $z$-measures: the families $\My_n^{z,z'}$ and $\Ms_n^\al$ can be characterized in a unified manner, see \cite{rozhkovskaya1997multiplicative}, \cite{Borodin1997}; see also \cite[\S4.1]{Petrov2010Pfaffian} for another characterization. On the other hand, most results about $\Ms_n^\al$ do not follow directly from the corresponding results about the $z$-measures. In this paper we aim to describe certain \emph{direct} combinatorial relations between $\My_n^{z,z'}$ and $\Ms_n^\al$, and, more general, between the Young and the Schur graphs. There are also other aspects in which $\My_n^{z,z'}$ and $\Ms_n^\al$ are directly related, e.g., at the level of correlation kernels of corresponding random point processes, see \cite[(7.17), (8.3)]{Petrov2010Pfaffian}, and \cite[Remark 6]{Petrov2010}.

\section{Doubling of shifted Young diagrams\\ and down transition probabilities}

Let $\la=(\la_1,\dots,\la_\ell)$ be a shifted Young diagram. By $\Dc\la$ let us denote its \emph{doubling}, i.e., the ordinary Young diagram with $2|\la|$ boxes which has Frobenius coordinates $(\la_1,\dots,\la_\ell\mid\la_1-1,\dots,\la_\ell-1)$ \cite[I.1]{Macdonald1995}. E.g., for $\la=(4,2)$ we have
\begin{equation*}
  \Dc\la\;=\ \begin{array}{|c|c|c|c|c|}
    \hline
    \ &. &. &. &. \\
    \hline 
    \ &\ &. &. \\
    \cline{1-4}
    \ &\ \\
    \cline{1-2}
    \ \\
    \cline{1-1}
  \end{array}
\end{equation*}
(the original shifted diagram is marked).
In this way, $\Dc$ defines an embedding of $\Sb$ into $\Yb$. Any ordinary Young diagram of the form $\Dc\la$ will be called \emph{$\Dc$-symmetric}. A finite path in the Young graph $\varnothing\nearrow\rho^{(1)}\nearrow\dots\nearrow\rho^{(k)}$ is called a \emph{$\Dc$-path} iff each Young diagram $\rho^{(2m)}$ is $\Dc$-symmetric. The next statement is straightforward:
\begin{lemma}\label{lemma:number_of_rho}
  Let $\mu\Nearrow\la$ be two shifted Young diagrams. If $\ell(\la)=\ell(\mu)$, there exist two ordinary diagrams $\rho^{(a),(b)}$ such that $\Dc\mu\nearrow\rho^{(a),(b)}\nearrow\Dc\la$. If $\ell(\la)=\ell(\mu)+1$, there is only one such ordinary diagram $\rho$. Consequently, for any $\la\in\Sb$ the number of $\Dc$-paths in $\Yb$ from $\varnothing$ to $\Dc\la$ is $2^{|\la|-\ell(\la)}g_\la$.
\end{lemma}

Note that the ambient structure of $\Dc$-paths in $\Yb$ defines certain edge multiplicities in $\Dc\Sb\subset\Yb$ (and, therefore, in $\Sb$): there are either one (if $\ell(\la)=\ell(\mu)+1$) or two (if $\ell(\la)=\ell(\mu)$) edges between shifted diagrams $\mu\Nearrow\la$. However, these new edge multiplicities give rise to the same down transition probabilities $p^\Da$ on $\Sb$ as before. Following \cite{Kerov1989}, one can say that the new edge multiplicities are \emph{equivalent} to the old ones.

\begin{prop}\label{prop:down_equality}
  In notation of Lemma \ref{lemma:number_of_rho}, if $\ell(\la)=\ell(\mu)$, one has $p^\Da(\la,\mu)= p^\da(\Dc\la,\rho^{(a)})+p^\da(\Dc\la,\rho^{(b)})$, and if $\ell(\la)=\ell(\mu)+1$, then $p^\Da(\la,\mu)= p^\da(\Dc\la,\rho)$.
\end{prop}
\begin{proof}
  Fix $\la\in\Sb_n$ and $\mu\in\Sb_{n-1}$ such that $\mu\Nearrow\la$. Assume that $\ell(\la)=\ell(\mu)$, the other case is similar. Down transition probabilities $p^\da$ on the Young graph allow to define a Markov chain going down from $\Dc\la$, i.e., a sequence of random ordinary Young diagrams $\Dc\la=\boldsymbol\varrho_0\searrow\boldsymbol\varrho_1\searrow\ldots\searrow\boldsymbol\varrho_{2n}= \varnothing$. For each $k$, the conditional distribution of $\boldsymbol\varrho_k$ given $\boldsymbol\varrho_{k-1}$ is governed by the transition kernel $p^\da$ from $\Yb_{2n-k+1}$ to $\Yb_{2n-k}$. In other words, this gives a measure on the set of all paths in $\Yb$ from $\varnothing$ to $\Dc\la$. By the very definition  of $p^\da$ (\ref{down_Y}), this measure is uniform over all such paths. Let $D$ denote the event that the path $(\boldsymbol\varrho_{2n}\nearrow\ldots\nearrow\boldsymbol\varrho_{0})$ from $\varnothing$ to $\Dc\la$ is a $\Dc$-path. Conditioning on the event $D$, we have a uniform measure over $\Dc$-paths. One clearly has
  \begin{equation}\label{down_p_proof}
    \prob(\boldsymbol\varrho_2=\Dc\mu,D)=
    \prob(\boldsymbol\varrho_1=\rho^{(a)},D)+
    \prob(\boldsymbol\varrho_1=\rho^{(b)},D). 
  \end{equation}
  In the left-hand side one has $(f_{\Dc\la})^{-1}$ times the number of $\Dc$-paths in $\Yb$ from $\varnothing$ to $\Dc\la$ which also go through $\Dc\mu$, and in the right-hand side the events $\{\boldsymbol\varrho_1=\rho^{(a),(b)}\}$ are independent of $D$, and $\prob(\boldsymbol\varrho_1=\rho^{(a),(b)})= p^\da(\Dc\la,\rho^{(a),(b)})$. Dividing (\ref{down_p_proof}) by $\prob(D)=2^{|\la|-\ell(\la)}g_\la/f_{\Dc\la}$, we get the desired identity.
\end{proof}

\section{Plancherel up transition probabilities}

Here we describe an identity for the Plancherel up transition probabilities $p^\ua_{\Ply}$ and $p^\Ua_{\Pls}$ which is ``dual'' to Proposition \ref{prop:down_equality}. The proof uses Kerov's interlacing coordinates of ordinary and shifted Young diagrams \cite{Kerov2000}, \cite{Olshanski2009}, \cite{petrov2009eng}. Let us recall necessary definitions and facts from these papers. 

For an ordinary Young diagram $\rho$, by $\xy_1,\dots,\xy_d$ and $\yy_1,\dots,\yy_{d-1}$ denote the contents of all boxes that can be added to or removed from $\rho$, respectively. It is known that these numbers interlace ($\xy_1<\yy_1<\xy_2<\dots<\yy_{d-1}<\xy_d$) and define $\rho$ uniquely. For example, for $\rho=(4,4,1)$ (see (\ref{YS_diagrams})) one has $d=3$, $(\xy_1,\xy_2,\xy_3)=(-3,-1,4)$, and $(\yy_1,\yy_2)=(-2,2)$. The Plancherel up transition probabilities for $\Yb$ arise as the following coefficients in the expansion as a sum of partial fractions:
\begin{equation*}
  \Rc^\ua(u;\rho):=\frac{(u-\yy_1)\dots(u-\yy_{d-1})}
  {(u-\xy_1)\dots(u-\xy_{d-1})(u-\xy_d)}=
  \sum_{s=1}^d \frac{p^\ua_{\Ply}(\rho;\rho+
  \fbox{$\xy_s$})}{u-\xy_s}.
\end{equation*}
Here $\rho+\fbox{$\xy_s$}$ means that we add to $\rho$ a box with content $\xy_s$.

The case of shifted diagrams is slightly more complicated, and in full detail it is explained in \cite[\S3]{petrov2009eng} (the arXiv version). Let $\la$ be a shifted Young diagram. Let $\ys_1,\dots,\ys_{k}$ denote the contents of all boxes that can be removed from $\la$. Let $\xs_1,\dots,\xs_k$ denote all the \emph{nonzero} contents of all boxes that can be added to $\la$. These contents also interlace ($\ys_1<\xs_1<\ys_2<\dots<\ys_{k}<\xs_k$) and define $\la$ uniquely. For example, for $\la=(5,3,2)$ (see (\ref{YS_diagrams})) one has $k=2$, $(\xs_1,\xs_2)=(3,5)$, and $(\ys_1,\ys_2)=(1,4)$. The Plancherel up transition probabilities for $\Sb$ arise as the following expansion coefficients:
\begin{equation*}
  \Rc^\Ua(v;\la):=
  \frac{(v-\ys_1(\ys_1+1))\dots(v-\ys_k(\ys_k+1))}
  {v(v-\xs_1(\xs_1+1))\dots(v-\xs_k(\xs_k+1))}
  =\sum_{\xs}
  \frac{p^\Ua_{\Pls}(\la;\la+\fbox{$\xs$})}
  {v-\xs(\xs+1)},
\end{equation*}
where the sum is taken over \emph{all} boxes which can be added to $\la$, and $\xs$ is the content of such a box (here it does not have to be nonzero).

The next fact is readily checked:
\begin{prop}\label{prop:up_Planch}
  For any $\la\in\Sb$, one has $(u-1)\cdot\Rc^\Ua(u(u-1);\la)=\Rc^\ua(u;\Dc\la)$. Consequently, in notation of Lemma \ref{lemma:number_of_rho}, $p^\Ua_{\Pls}(\mu,\la)= p^\ua_\Ply(\Dc\mu,\rho^{(a)})+p^\ua_\Ply(\Dc\mu,\rho^{(b)})$ for  $\ell(\la)=\ell(\mu)$, and $p^\Ua_{\Pls}(\mu,\la)= p^\ua_\Ply(\Dc\mu,\rho)$ otherwise.
\end{prop}

Proposition \ref{prop:down_equality} can also be proved using the above rational functions because the down transition probabilities essentially arise as coefficients of expansions of $1/\Rc^\ua(u;\rho)$ and $1/(v\cdot \Rc^\Ua(v;\la))$.

\section{Up transition probabilities for $\My_n^{z,z'}$ and $\Ms_n^\al$}
\label{section:up_z_al}

By suitable choice of the parameters $z,z'$ of the $z$-measures on ordinary partitions, one can get an analogue of Proposition \ref{prop:up_Planch} for the deformed coherent systems $\My_n^{z,z'}$ and $\Ms_n^\al$, which is the main result of the present note. Set $\nu(\al):=\frac12\sqrt{1-4\al}$. From (\ref{up_Z}), (\ref{up_mult}) and Proposition \ref{prop:up_Planch} we have:

\begin{prop}\label{prop:up_Z}
  Let $z(\al)=\nu(\al)-\frac12$, $z'(\al)=-\nu(\al)-\frac12$ (note that these parameters are admissible for the $z$-measures). In notation of Lemma \ref{lemma:number_of_rho}, for $\ell(\la)=\ell(\mu)$ one has $p^\Ua_{\al}(\mu,\la)= p^\ua_{z(\al),z'(\al)}(\Dc\mu,\rho^{(a)})+p^\ua_{z(\al),z'(\al)}(\Dc\mu,\rho^{(b)})$, and if $\ell(\la)=\ell(\mu)+1$, then one has $p^\Ua_{\al}(\mu,\la)= p^\ua_{z(\al),z'(\al)}(\Dc\mu,\rho)$.
\end{prop}

  Now one can explain how the random growth processes for the measures $\Ms^\al_n$ and $\My_n^{z(\al),z'(\al)}$ are related. Indeed, to grow a random \emph{shifted} Young diagram $\boldsymbol\la$ with $n$ boxes distributed according to $\Ms^\al_n$, one should start the growth process on the \emph{Young graph} from $\varnothing$ which evolves as follows:

  $\bullet$ at each \emph{even} step add a box to the ordinary diagram according to the probabilities $p^\ua_{z(\al),z'(\al)}$ (this is a random procedure);

  $\bullet$ at each \emph{odd} step add the unique box to the current ordinary diagram so that it again becomes $\Dc$-symmetric (this is a deterministic procedure).
  
  In this way the growth process on the Young graph goes along a $\Dc$-symmetric path, and after $2n$ steps it reaches a random ordinary Young diagram $\Dc\boldsymbol\la\in\Yb_{2n}$, where $\boldsymbol\la\in\Sb_n$ is distributed according to $\Ms^\al_n$. One may call this the \emph{forced $\Dc$-symmetrization} of the old growth process (\ref{up_Z}) on the Young graph: the growing ordinary Young diagram is forced to be $\Dc$-symmetric at every step at which it is possible.

\section{Schur measures and an analogue for shifted diagrams}

Both families of measures that we consider can be interpreted through certain specializations of Schur symmetric functions $s_\tau$, $\tau\in\Yb$ \cite[I.3]{Macdonald1995}. For the $z$-measures one has \cite{okounkov2001infinite}
\begin{equation*}
  \My_{n}^{z,z'}(\rho)=
  \frac{n!}{(zz')_n}
  s_\rho(\underbrace{1,\dots,1}_{z\text{ times}})
  s_\rho(\underbrace{1,\dots,1}_{z'\text{ times}}),
  \qquad \rho\in\Yb_n
\end{equation*}
(here $(\cdots)_n$ denotes the Pochhammer symbol). For the measures $\Ms_n^\al$ one can show that (see also \cite[\S2.6]{Petrov2010})
\begin{equation*}
  \Ms_n^\al(\la)=
  \frac{(-1)^nn!}{(\al/2)_n}
  s_{\Dc\la}
  (\underbrace{1,1,\dots,1,1}_{\nu(\al)-\frac12\text{ times}}),
  \qquad \la\in\Sb_n.
\end{equation*}
Such measures were first considered in \cite[Thm 7.1]{Rains2000}. 

From the above two formulas one sees that the weights $\{\Ms_n^\al(\la)\}_{\la\in\Sb_n}$ are proportional (with a coefficient depending only on $n$) to square roots of the weights $\{\My_{2n}^{z(\al),z'(\al)}(\Dc\la)\}_{\la\in\Sb_n}$. (Alternatively, this can be seen from \S\ref{section:up_z_al} and the multiplicative nature of our measures \cite{rozhkovskaya1997multiplicative}, \cite{Borodin1997}.) This property can easily be reformulated in probabilistic terms, but it does not seem to provide a direct way of obtaining properties of $\Ms_n^\al$ from the corresponding properties of the $z$-measures.

\providecommand{\bysame}{\leavevmode\hbox to3em{\hrulefill}\thinspace}
\providecommand{\MR}{\relax\ifhmode\unskip\space\fi MR }
\providecommand{\MRhref}[2]{%
  \href{http://www.ams.org/mathscinet-getitem?mr=#1}{#2}
}
\providecommand{\href}[2]{#2}

\end{document}